\documentclass[12pt, a4paper,reqno]{amsart}
\usepackage{hyperref}
\hypersetup{colorlinks=true, bookmarks=true,citecolor=blue}
\usepackage{amssymb,amstext,amsmath,amscd,amsthm,amsfonts,enumerate,graphicx,latexsym,color,stmaryrd,bbm}
\usepackage[all]{xy}
\usepackage{comment}

\allowdisplaybreaks
\newtheorem{thm}{Theorem}[section]
\newtheorem*{thm*}{Theorem}

\newtheorem{cor}[thm]{Corollary}
\newtheorem{lem}[thm]{Lemma}
\newtheorem{prop}[thm]{Proposition}
\theoremstyle{definition}
\newtheorem{conv}[thm]{Convention}
\newtheorem{dfn}[thm]{Definition}
\newtheorem*{dfn*}{Definition}
\newtheorem{rem}[thm]{Remark}

\newtheorem*{conj*}{Conjecture}

\theoremstyle{remark}

\newtheorem*{claim*}{Claim}

\numberwithin{equation}{thm}
\newcommand{\ilim}[1][]{\mathop{\varinjlim}\limits_{#1}}

\def\LL{\mathbb{L}}

\def\PP{\mathbb{P}}

\def\ZZ{\mathbb{Z}}

\newcommand{\rZ}{\mathrm{Z}}

\newcommand{\sZ}{\mathsf{Z}}

\newcommand{\cE}{\mathcal{E}}
\newcommand{\cF}{\mathcal{F}}

\newcommand{\cK}{\mathcal{K}}

\newcommand{\cO}{\mathcal{O}}
\newcommand{\cP}{\mathcal{P}}
\newcommand{\cQ}{\mathcal{Q}}
\newcommand{\cS}{\mathcal{S}}
\newcommand{\cT}{\mathcal{T}}
\newcommand{\cU}{\mathcal{U}}
\newcommand{\cV}{\mathcal{V}}

\newcommand{\cX}{\mathcal{X}}

\def\dpf{\mathrm{D^{pf}}}
\def\dub{\mathrm{D}}

\def\Mod{\operatorname{\mathrm{Mod}}}

\def\End{\mathrm{End}}
\def\Hom{\mathrm{Hom}}

\def\lnil{\mathrm{lnil}}
\def\lred{\mathrm{lred}}
\def\ltensor{\otimes^{\LL}}
\def\id{\mathrm{id}}
\def\red{\mathrm{red}}
\def\res{\mathrm{res}}

\def\Spec{\mathrm{Spec}}
\def\tSpec{\mathrm{Spec}_\triangle}
\def\ttSpec{\mathrm{Spec}_\otimes}

\def\Th{\mathbf{Th}}

\def\ol{\overline}

\def\nat{\natural}

\def\one{\mathbf{1}}

\title[Triangular Spectra and Derived categories]{Triangular spectra and their applications to derived categories of noetherian schemes}

\author{Hiroki Matsui}
\address{Department of Mathematical Sciences, Faculty of Science and Technology, Tokushima University, 2-1 Minamijyousanjima-cho, Tokushima 770-8506, Japan}
\email{hmatsui@tokushima-u.ac.jp}

\thanks{2020 {\em Mathematics Subject Classification.} 14A15, 14F08, 14H52, 18G80}

\thanks{{\em Key words and phrases.} Balmer spectrum, elliptic curve, noetherian scheme, perfect derived category, tensor triangulated category, triangular spectrum, triangulated category} 

\thanks{The author was partly supported by JSPS Grant-in-Aid for Early-Career Scientists 22K13894.}

\begin{document}

\begin{abstract}
In recent work, for a triangulated category $\cT$, the author introduced a topological space $\tSpec(\cT)$ which we call the triangular spectrum of $\cT$ as a tensor-free analog of the Balmer spectrum for a tensor triangulated category.
In this paper, we use the triangular spectrum to reconstruct a noetherian scheme $X$ from its perfect derived category $\dpf(X)$. 
As an application, we give an alternative proof of the Bondal-Orlov-Ballard reconstruction theorem in the special case (when both varieties have ample or anti-ample canonical bundles).
Moreover, we define the structure sheaf on $\tSpec(\cT)$ and compare the triangular spectrum and the Balmer spectrum as ringed spaces.
\end{abstract}

\maketitle
\section{Introduction}

In this paper, we consider the {\it reconstruction problem} of noetherian schemes from their {\it perfect derived categories}, asking whether a triangle equivalence between perfect derived categories $\dpf(X) \cong \dpf(Y)$ implies an isomorphism $X \cong Y$ of noetherian schemes?
Many authors have studied this kind of reconstruction problem well; see  \cite{Ball, Bal02, BO, Cal, Sp}.
It is well-known that affine noetherian schemes are reconstructed using the triangulated category structures from their perfect derived categories.
By contrast, this reconstruction problem fails for non-affine noetherian schemes in general.
For example, Mukai \cite{Muk} proved that an abelian $A$ and its dual $A^\vee$ (which are not isomorphic in general) have the equivalent perfect derived categories. 

Therefore, the triangulated category structure is insufficient for reconstructing $X$ from $\dpf(X)$. 
Balmer proved that $X$ can be reconstructed from $\dpf(X)$ using the tensor triangulated category structure as follows.
For an essentially small tensor triangulated category $(\cT, \otimes, \one)$, Balmer defined the ringed space $\ttSpec(\cT)$, which we call the {\it Balmer spectrum} of $(\cT, \otimes, \one)$.
In \cite{Bal05}, he proved that there is an isomorphism
$$
X \cong \ttSpec(\dpf(X))
$$	
for the tensor triangulated category $(\dpf(X), \ltensor_{\cO_X}, \cO_X)$.
This isomorphism shows that $X$ is reconstructed from $\dpf(X)$ using the tensor triangulated category structure.
Balmer spectra allow us to study tensor triangulated categories via algebro-geometric methods.
This theory is called {\it tensor triangular geometry} and has been actively studied in various fields of mathematics.
Recently, the author \cite{Mat} introduced the topological space $\tSpec(\cT)$ for a triangulated category $\cT$ to generalize tensor triangular geometry to arbitrary essentially small triangulated categories.
We call $\tSpec(\cT)$ the {\it triangular spectrum} of $\cT$.
For the perfect derived category $\dpf(X)$ of a noetherian scheme $X$, it is shown in \cite{Mat} that there is an immersion
$$
X \hookrightarrow \tSpec(\dpf(X))
$$
of topological spaces. 

One of the most famous reconstruction results is due to Bondal-Orlov \cite{BO} and Ballard \cite{Ball}, which states that for Gorenstein projective varieties $X_1$ and $X_2$ over a field $k$ with ample or anti-ample canonical bundles, $X_1$ and $X_2$ are isomorphic as varieties whenever their perfect derived categories are equivalent as $k$-linear triangulated categories.
Although they just assume either $X_1$ or $X_2$ has an ample or anti-ample canonical bundle, we call above the {\em Bondal-Orlov-Ballard reconstruction theorem}.
As an application of triangular spectra, we prove the following result, which generalizes the Bondal-Orlov-Ballard reconstruction theorem.

\begin{thm}\rm{(Theorem \ref{main})}.\label{intro1}
Let $X_1$ and $X_2$ be noetherian schemes, and $\Phi: \dpf(X_1) \xrightarrow{\cong} \dpf(X_2)$ be a triangle equivalence.
Assume there is a line bundle $L_i$ on $X_i$ for $i=1,2$ satisfying the following conditions:
\begin{enumerate}[\rm(i)]
\item
$L_i$ is ample or anti-ample for $i=1,2$.
\item
There is an isomorphism
$$
\Phi(F \ltensor_{\cO_{X_1}} L_1) \cong \Phi(F) \ltensor_{\cO_{X_2}} L_2
$$
for any $F \in \dpf(X_1)$.
\end{enumerate}
Then $(X_1)_{\red}$ and $(X_2)_{\red}$ are isomorphic as schemes.
\end{thm}

\noindent
Actually, the reconstruction of underlying topological spaces essentially appeared in \cite{HO}.
Although they use the result of Koll\'ar-Lieblich-Olsson-Sawin \cite{KLOS} to reconstruct structure sheaves, we use the {\it center} of a triangulated category, which plays an important role in this paper.


So far, we have considered the triangular spectrum $\tSpec(\cT)$ as a topological space.
In this paper, we will introduce the structure sheaf on $\tSpec(\cT)$ for a triangulated category $\cT$, which makes $\tSpec(\cT)$ a ringed space.
The following second main theorem allows us to compare the Balmer spectrum and the triangular spectrum as ringed spaces:

\begin{thm}\rm{(Theorem \ref{main2})}.\label{intro2} 
Let $(\cT, \otimes, \one)$ be an idempotent complete, rigid, and locally monogenic tensor triangulated category. 
Assume further that $\ttSpec(\cT)$ is noetherian.
\begin{enumerate}[\rm(1)]
\item
There is an inclusion $\ttSpec(\cT) \subseteq \tSpec(\cT)$.
\item
There is a morphism 
$$
i: \ttSpec(\cT)_\red \to \tSpec(\cT)
$$
of ringed spaces satisfying the following conditions:
\begin{enumerate}[\rm(a)]
\item
The map of underlying topological spaces is the inclusion $\ttSpec(\cT) \subseteq \tSpec(\cT)$.
\item
$i$ is an open immersion of ringed spaces whenever $\ttSpec(\cT)$ is an open subset of $\tSpec(\cT)$.
\item
$i$ is an isomorphism of ringed spaces if $\cT$ is classically generated by 
$\one$.
\end{enumerate}
	
\end{enumerate}
\end{thm}

\noindent
Here, a tensor triangulated category $(\cT, \otimes, \one)$ is said to be {\it rigid} if it is {\it closed} and every object is {\it strongly dualizable}. It is said to be {\it locally monogenic} if it is locally classically generated by the unit object; see Section 3 for details.  

The assumptions in Theorem \ref{intro2} are satisfied for the perfect derived categories of noetherian schemes.
Applying this theorem to such tensor triangulated categories, we obtain the following result:

\begin{thm}\rm{(Corollaries \ref{imm}, \ref{p1}, and \ref{elp})}. \label{intro3}
\begin{enumerate}[\rm(1)]
\item
Let $X$ be a noetherian quasi-affine scheme.
Then there is an isomorphism
$$
\tSpec(\dpf(X)) \cong X_\red 
$$
of ringed spaces.
\item
Let $\PP^1$ be the projective line over a field $k$.
Then there is an isomorphism
$$
\tSpec(\dpf(\PP^1)) \cong \PP^1 \sqcup \left( \bigsqcup_{n \in \ZZ} \Spec (k)\right)
$$
of ringed spaces.
\item
Let $E$ be an elliptic curve over an algebraically closed field. Then there is an isomorphism
$$
\tSpec(\dpf(E)) \cong E \sqcup \left(\bigsqcup_{(r,d) \in I} E_{r,d}\right)
$$
of ringed spaces.
Here, $I:=\{(r,d) \in \ZZ^{2} \mid r >0, \gcd(r,d) = 1\}$, and $E_{r,d}$ is a copy of $E$ for each $(r,d) \in I$.
\end{enumerate}
\end{thm}

\noindent 
Matsui \cite{Mat} and Hirano-Ouchi \cite{HO} proved that there are homeomorphisms between the underlying topological spaces of ringed spaces in Theorem \ref{intro3}(1),(2) and Theorem \ref{intro3}(3), respectively.
Therefore, Theorem \ref{intro3} extends their result to isomorphisms of ringed spaces.
Moreover, Theorem \ref{intro3}(1),(3) shows that reduced quasi-affine schemes and elliptic curves are reconstructed from their perfect derived categories using only triangulated category structure (without tensor structure).

This paper is organized as follows.
In Section 2, we give the definitions of the center and the triangular spectra of a triangulated category, which play a central role throughout this paper. 
In Section 3, we prove Theorem \ref{intro1} and give its applications, including the Bondal-Orlov-Ballard reconstruction theorem.
In Section 4, we introduce the structure sheaf on the triangular spectrum and prove Theorem \ref{intro2}.
We apply this result to the perfect derived categories of noetherian schemes and deduce Theorem \ref{intro3}.
\section{Preliminaries}

In this section, we recall basic definitions and properties for later use.
We begin with our convention.

\begin{conv}
\begin{enumerate}[\rm(1)]
\item
Throughout this paper, we assume that all triangulated categories are essentially small and that all subcategories are full. 

Let $\cT$ be a triangulated category. 
A {\it thick subcategory} of $\cT$ is a subcategory closed under direct summands, shifts, and extensions.
The set $\Th(\cT)$ of all thick subcategories forms a lattice with respect to the inclusion relation. 
For $M \in \cT$, denote by $\langle M \rangle$ the smallest thick subcategory of $\cT$ containing $M$. 
\item
For a noetherian scheme $X$, a complex $F$ of $\cO_X$-modules is said to be {\it perfect} if, for any $x \in X$, there is an open neighborhood $U \subseteq X$ of $x$ such that the restriction $F|_U$ is quasi-isomorphic to a bounded complex of locally free sheaves of finite rank.
Denote by $\dpf(X)$ the derived category of perfect complexes on $X$.
We call it the {\it perfect derived category} of $X$.
\end{enumerate}
\end{conv}

\subsection{Center of triangulated categories}

We recall the definition and basic properties of the center of a triangulated category; see \cite{KY} for details. 
 
\begin{dfn}
Let $\cT$ be a triangulated category.
\begin{enumerate}[\rm(1)]
\item
The {\it center} $\rZ(\cT)$ of $\cT$ is the set of natural transformations $\eta: \id_\cT \to \id_{\cT}$ with $\eta[1] = [1]\eta$.	
The composition of natural transformations makes $\rZ(\cT)$ a commutative ring.
\item
We say that an element $\eta \in \rZ(\cT)$ is {\it locally nilpotent} if $\eta_M$ is a nilpotent element of the endomorphism ring $\End_{\cT}(M)$ for each $M \in \cT$.
We shall denote by $\rZ(\cT)_{\mathrm{lnil}}$ the ideal of $\rZ(\cT)$ consisting of locally nilpotent elements and by $\rZ(\cT)_{\mathrm{lred}} := \rZ(\cT)/\rZ(\cT)_{\mathrm{lnil}}$ the quotient ring.
\end{enumerate}
\end{dfn}

Let $\Phi: \cT \to \cT'$ be an exact functor between triangulated categories. 
It seems to be not known whether $\Phi$ induces a morphism between $\rZ(\cT)_{\lred}$ and $\rZ(\cT')_{\lred}$ in general. 
Let us give several functoriality results of $\rZ(-)_{\lred}$ under certain functors following \cite{KY}.
 
We say that an exact functor $\Phi: \cT \to \cT'$ is {\it dense} if, for any $M' \in \cT'$, there are $M \in \cT$ and $N' \in \cT'$ such that $\Phi(M) \cong M' \oplus N'$.
For example, the idempotent completion functor $\iota: \cT \to \cT^\nat$ of $\cT$ is fully faithful and dense; see \cite{BS}. 

\begin{lem}\label{cntidm}
Let $\Phi: \cT \to \cT'$ be a fully faithful dense exact functor. Then there is an isomorphism 
$
\Phi^* : \rZ(\cT') \overset{\cong}{\to} \rZ(\cT),
$
where $\Phi^*(\eta)_M = \Phi^{-1}(\eta_{\Phi(M)})$ for $M \in \cT$.
Moreover, this isomorphism induces an isomorphism
$
\ol{\Phi^*} : \rZ(\cT')_{\mathrm{lred}} \overset{\cong}{\to} \rZ(\cT)_{\mathrm{lred}}.
$
\end{lem}

\begin{proof}
We note that for any object $M' \in \cT'$, there is an object $M \in \cT$ and an isomorphism $M' \oplus M'[1] \cong \Phi(M)$; see \cite[(3.2) in the proof of Proposition 3.13]{Bal05}.

As $\Phi: \cT \to \cT'$ is a fully faithful exact functor, it induces a homomorphism $\Phi^*: \rZ(\cT') \to \rZ(\cT)$ by \cite[Proposition 2.3(1)]{KY}.
First, we prove that this homomorphism is an isomorphism.
Let $\eta \in \rZ(\cT')$ with $\Phi^*(\eta) = 0$.
For any $M' \in \cT'$, take an object $M \in \cT$ and an isomorphism $M' \oplus M'[1] \cong \Phi(M)$.
Then we obtain $\Phi^*(\eta)_M = \Phi^{-1}(\eta_{\Phi(M)}) \cong \Phi^{-1}(\eta_{M'}) \oplus \Phi^{-1}(\eta_{M'})[1]$.
Therefore, $\Phi^*(\eta)_M =0$ implies $\Phi^{-1}(\eta_{M'})=0$ and hence $\eta_{M'} = 0$.
This shows that $\Phi^*: \rZ(\cT') \to \rZ(\cT)$ is injective. 
To show that $\Phi^*: \rZ(\cT') \to \rZ(\cT)$ is surjective, fix an element $\delta \in \rZ(\cT)$ and construct $\eta \in \rZ(\cT')$ with $\Phi^*(\eta)=\delta$. 
For any $M' \in \cT'$, take an object $M \in \cT$ and an isomorphism $\varphi: \Phi(M) \xrightarrow{\cong} M' \oplus M'[1]$.
Then we define the morphism $\eta_{M'}: M' \to M'$ as the composition
$$
M' \overset{\left(\begin{smallmatrix}1 \\ 0 \end{smallmatrix}\right)}{\to} M' \oplus M'[1] \xrightarrow{\varphi^{-1}} \Phi(M) \xrightarrow{\Phi(\delta_M)} \Phi(M) \xrightarrow{\varphi} M' \oplus M'[1] \overset{\left(\begin{smallmatrix}1 & 0 \end{smallmatrix}\right)}{\to} M'.
$$
This morphism $\eta_{M'}$ does not depend on the choices of $M$ and $\varphi$.
Indeed, take an object $N \in \cT$ and an isomorphism $\psi: \Phi(N) \xrightarrow{\cong} M' \oplus M'[1]$.
Since $\Phi$ is full, there is a morphism $f: M \to N$ in $\cT$ such that $\Phi(f) = \psi^{-1}\varphi$.
Then we have the equalities
$$
\varphi \Phi(\delta_M) \varphi^{-1} = \psi \Phi(f) \Phi(\delta_M) \varphi^{-1} = \psi \Phi(\delta_N) \Phi(f) \varphi^{-1} = \psi \Phi(\delta_N)\psi^{-1}.
$$ 
For this reason, $\eta_{M'}$ is well-defined.
For a morphism $g: M' \to N'$ in $\cT'$, we prove $g \eta_{M'} = \eta_{N'}g$.
Take objects $M, N \in \cT$ and isomorphisms $\varphi: \Phi(M) \xrightarrow{\cong} M' \oplus M'[1]$, $\psi: \Phi(N) \xrightarrow{\cong} N' \oplus N'[1]$ in $\cT'$.
Since $\Phi$ is full, there is a morphism $f: M \to N$ such that $\Phi(f) = \psi^{-1}\left(\begin{smallmatrix}g & 0 \\ 0 & g[1] \end{smallmatrix}\right)\varphi$.
Then we get the equalities
\begin{align*}
g \eta_{M'} &= g\left(\begin{smallmatrix}1 & 0 \end{smallmatrix}\right) \varphi \Phi(\delta_M)\varphi^{-1}\left(\begin{smallmatrix}1 \\ 0 \end{smallmatrix}\right) \\
&= \left(\begin{smallmatrix}1 & 0 \end{smallmatrix}\right) \left(\begin{smallmatrix}g & 0 \\ 0 & g[1] \end{smallmatrix}\right)\varphi \Phi(\delta_M)\varphi^{-1}\left(\begin{smallmatrix}1 \\ 0 \end{smallmatrix}\right) \\
&= \left(\begin{smallmatrix}1 & 0 \end{smallmatrix}\right) \psi \Phi(f)\Phi(\delta_M)\varphi^{-1}\left(\begin{smallmatrix}1 \\ 0 \end{smallmatrix}\right) \\
&= \left(\begin{smallmatrix}1 & 0 \end{smallmatrix}\right) \psi \Phi(\delta_N)\Phi(f)\varphi^{-1}\left(\begin{smallmatrix}1 \\ 0 \end{smallmatrix}\right) \\
&= \left(\begin{smallmatrix}1 & 0 \end{smallmatrix}\right) \psi \Phi(\delta_N)\psi^{-1}\left(\begin{smallmatrix}g & 0 \\ 0 & g[1] \end{smallmatrix}\right)\left(\begin{smallmatrix}1 \\ 0 \end{smallmatrix}\right) \\
&= \left(\begin{smallmatrix}1 & 0 \end{smallmatrix}\right) \psi \Phi(\delta_N)\psi^{-1}\left(\begin{smallmatrix}1 \\ 0 \end{smallmatrix}\right) g = \eta_{N'}g.
\end{align*}
This shows that $\eta: \id_{\cT'} \to \id_{\cT'}$ is a natural transformation.
Moreover, one can easily see from the definition that $\eta[1]=[1]\eta$ holds and hence $\eta \in \rZ(\cT')$.
Finally, we will check the equality $\Phi^*(\eta) = \delta$. For an object $M \in \cT$, we can take the canonical isomorphism $\Phi(M \oplus M[1]) \cong \Phi(M) \oplus \Phi(M)[1]$.
Using this isomorphism, we have a commutative diagram
$$
\xymatrix{
\Phi(M) \ar[r]^-{\left(\begin{smallmatrix}1 \\ 0 \end{smallmatrix}\right)} \ar[d]_{\Phi(\delta_M)} & \Phi(M) \oplus \Phi(M)[1] \ar@{}[r]|-*[@]{\cong}\ar[d]^{\left(\begin{smallmatrix}\Phi(\delta_M) & 0 \\ 0 & \Phi(\delta_M)[1] \end{smallmatrix}\right)} & \Phi(M \oplus M[1]) \ar[d]^{\Phi(\delta_{M \oplus M[1]})} \\
\Phi(M) & \Phi(M) \oplus \Phi(M)[1] \ar[l]^-{\left(\begin{smallmatrix}1 & 0 \end{smallmatrix}\right)} \ar@{}[r]|-*[@]{\cong} & \Phi(M \oplus M[1]) 
}
$$
and this means that $\eta_{\Phi(M)} = \Phi(\delta_M)$. 
Therefore, $\Phi^*(\eta)_M = \Phi^{-1}(\eta_{\Phi(M)}) = \delta_M$.
Thus, we conclude that $\Phi^*(\eta)=\delta$.

We finish the proof by checking that the first isomorphism induces the second one.
From the first isomorphism, the induced homomorphism $\ol{\Phi^*} : \rZ(\cT')_{\mathrm{lred}} \twoheadrightarrow \rZ(\cT)_{\mathrm{lred}}$ is surjective.
Pick an element $\eta \in \rZ(\cT')$ with $\Phi^*(\eta) \in \rZ(\cT)_{\mathrm{lnil}}$. 
For any object $M \in \cT$, there is an integer $n>0$ such that $\Phi^{-1}(\eta_{\Phi(M)}^n) = \Phi^{-1}(\eta_{\Phi(M)})^n = \Phi^*(\eta)_M^n = 0$ as $\Phi^*(\eta)$ is locally nilpotent.
Hence $\eta_{\Phi(M)}^n = 0$.
Because each object $M' \in \cT'$ is a direct summand of $\Phi(M)$ for some $M \in \cT$, we get $\eta \in \rZ(\cT')_{\mathrm{lred}}$.
This shows the injectivity of $\ol{\Phi^*} : \rZ(\cT')_{\mathrm{lred}} \twoheadrightarrow \rZ(\cT)_{\mathrm{lred}}$.
\end{proof}

For thick subcategories $\cU \subseteq \cV$ of $\cT$, there is a unique exact functor $Q_{\cV/\cU}: \cT/\cU \to \cT/\cV$ such that $Q_{\cV/\cU} \circ Q_{\cU} = Q_{\cV}$, where $Q_{\cU}:\cT \to \cT/\cU$ and $Q_{\cV}:\cT \to \cT/\cV$ are the canonical functors. 

\begin{lem}\label{cntquot}
Let $\cU \subseteq \cV$ be thick subcategories of $\cT$.
Then the canonical functor $Q_{\cV/\cU}: \cT/\cU \to \cT/\cV$ induces a ring homomorphism
$
(Q_{\cV/\cU})_*: \rZ(\cT/\cU) \to \rZ(\cT/\cV)
$
where $(Q_{\cV/\cU})_*(\eta)_{Q_{\cV}(M)} = Q_{\cV/\cU}(\eta_{Q_{\cU}(M)})$ for $M \in \cT$.
Moreover, this homomorphism induces a homomorphism
$
\ol{(Q_{\cV/\cU})_*}: \rZ(\cT/\cU)_{\mathrm{lred}} \to \rZ(\cT/\cV)_{\mathrm{lred}} 
$.
\end{lem}

\begin{proof}
It follows from \cite[Proposition 2.3.1]{Ver} that $\cV/\cU$ is a thick subcategory of $\cT/\cU$ and that $Q_{\cV/\cU}$ induces a triangle equivalence
$$
\Phi: (\cT/\cU)/(\cV/\cU) \xrightarrow{\cong} \cT/ \cV
$$
such that $\Phi \circ Q = Q_{\cV/\cU}$, where $Q: \cT/\cU \to (\cT/\cU)/(\cV/\cU)$ is the canonical functor.
By Lemma \ref{cntidm} and \cite[Proposition 2.3(2)]{KY}, we obtain the homomorphism
$$
(Q_{\cV/\cU})_* : \rZ(\cT/\cU) \xrightarrow{Q_*} \rZ((\cT/\cU)/(\cV/\cU)) \underset{\cong} {\xrightarrow{(\Phi^*)^{-1}}} \rZ(\cT/\cV).
$$
The second statement is clear from the description $(Q_{\cV/\cU})_*(\eta)_{Q_{\cV}(M)} = Q_{\cV/\cU}(\eta_{Q_{\cU}(M)})$.
\end{proof}

\subsection{Spectra of triangulated categories}

In this subsection, let us recall the two kinds of spectra introduced by Balmer \cite{Bal05} and Matsui \cite{Mat}.

Let $\cT$ be a triangulated category.
We set
$$
\sZ(\cE) := \{\cX \in \Th(\cT) \mid \cX \cap \cE = \emptyset\}
$$
for a subcategory $\cE$ of $\cT$.
Then we can easily see the following properties:
\begin{enumerate}[\rm(i)]
\item
$\sZ(\cT) = \emptyset$ and $\sZ(\emptyset) = \Th(\cT)$.
\item
$\bigcap_{i \in I} \sZ(\cE_i) = \sZ(\bigcup_{i \in I} \cE_i)$. 
\item	
$\sZ(\cE) \cup \sZ(\cE') = \sZ(\cE \oplus \cE')$, where $\cE \oplus \cE' := \{M \oplus M' \mid M \in \cE, M' \in \cE'\}$. 
\end{enumerate}
Therefore, we can define the topology on $\Th(\cT)$ whose closed subsets are of the form $\sZ(\cE)$ for some $\cE \subseteq \cT$.

A {\it tensor triangulated category} is a triple $(\cT, \otimes, \one)$ consisting of a triangulated category $\cT$ together with a symmetric monoidal structure $(\otimes, \one)$ such that the bifunctor $\otimes: \cT \times \cT \to \cT$ is exact in each variable.
 
\begin{dfn}(\cite[Definition 2.1]{Bal05}).
Let $(\cT, \otimes, \one)$ be a tensor triangulated category.
A proper thick subcategory $\cP \subsetneq \cT$ is said to be a {\it prime thick ideal} if it satisfies the following conditions:  
\begin{itemize}
\item (ideal)
$M \otimes N \in \cP$ holds for any $M \in \cT$ and $N \in \cP$,
\item (prime)
$M \otimes N \in \cP$ implies $M \in \cP$ or $N \in \cP$.
\end{itemize}
Denote by $\ttSpec(\cT)$ the set of prime thick ideals of $\cT$ together with the induced topology of $\Th(\cT)$.
We call $\ttSpec(\cT)$ the {\it Balmer spectrum} or the {\it tensor-triangular spectrum} of $(\cT, \otimes, \one)$.
\end{dfn}

Balmer \cite{Bal05} also defined the structure sheaf on $\ttSpec(\cT)$.
For an open subset $U \subseteq \ttSpec(\cT)$, set $\cT^U := \bigcap_{\cP \in U} \cP$.
We define the tensor triangulated category $\cT(U)$ by
$$
\cT(U) := \left(\cT/\cT^U\right)^{\nat}.
$$
Denote by $\one_U$ the image of $\one$ under the canonical functor $\cT \to \cT(U)$.

\begin{dfn}(\cite[Definition 6.1]{Bal05}).
We denote by $\cO_{\otimes, \cT}$ the sheafification of the presheaf $\cO^p_{\otimes, \cT}$ of commutative rings given by the assignment
$$
U \mapsto \cO^p_{\otimes, \cT}(U) := \End_{\cT(U)}(\one_U).
$$
We simply write $\ttSpec(\cT)$ for the ringed space $(\ttSpec(\cT), \cO_{\otimes, \cT})$.
\end{dfn}

Using the Balmer spectrum, Balmer \cite{Bal05} proved that any noetherian scheme $X$ could be reconstructed from the tensor triangulated category $(\dpf(X), \ltensor_{\cO_{X}}, \cO_X)$.

\begin{thm}$($\cite[Theorem 6.3(a)]{Bal05}$)$. \label{balm}
For a noetherian scheme $X$, there is an isomorphism
$$
\cS_X: X \xrightarrow{\cong} \ttSpec(\dpf(X))
$$	
of ringed spaces, where the map of underlying topological spaces is given by
$$
x \mapsto \cS_X(x) := \{\cF \in \dpf(X) \mid \cF_x \cong 0 \mbox{ in } \dpf(\cO_{X,x})\}.
$$
\end{thm}

Next, let us recall the triangular spectrum of a triangulated category of $\cT$.

\begin{dfn}(\cite[Definitions 2.2 and 2.4]{Mat}).
Let $\cT$ be a triangulated category.
A proper thick subcategory $\cP \subsetneq \cT$ is called a {\it prime thick subcategory} if the partially ordered set $\{\cX \in \Th(\cT) \mid \cP \subsetneq \cX \}$ has the smallest element. 
Denote by $\tSpec(\cT)$ the set of prime thick subcategories of $\cT$ together with the induced topology of $\Th(\cT)$.
We call $\tSpec(\cT)$ the {\it triangular spectrum} of $\cT$.
\end{dfn}

\begin{rem}
In \cite{Mat}, we call $\cP$ a prime thick subcategory if $\{\cX \in \Th(\cT) \mid \cP \subsetneq \cX \}$ has a {\em unique minimal element}.
However, correctly it should have the {\em smallest element} as above.
Actually, the proofs in \cite{Mat} are done using the above definition (``unique minimal element” in \cite[Lemma 2.15 and Proposition 4.7]{Mat} should be changed to ``smallest element”). 
Therefore, we changed the definition of a prime thick subcategory.

\end{rem}

Let $\Phi: \cT \to \cT'$ be an exact functor.
Since $\Phi^{-1}(\cX):= \{M \in \cT \mid \Phi(M) \in \cT'\}$ is a thick subcategory of $\cT$ for each thick subcategory $\cX \subseteq \cT'$, we have an order-preserving map
$$
\Phi^*: \Th(\cT') \to \Th(\cT), \quad \cX \mapsto \Phi^{-1}(\cX).
$$
This map restricts to continuous maps between triangular spectra for fully faithful dense exact functors and quotient functors.

\begin{lem}$($\cite[Proposition 2.11]{Mat}$)$. \label{idm}
Let $\Phi: \cT \to \cT'$ be a fully faithful dense exact functor.
Then the map $\Phi^*: \Th(\cT') \to \Th(\cT)$ restricts to a homeomorphism
$$
\Phi^*: \tSpec(\cT') \overset{\cong}{\to} \tSpec(\cT).
$$ 
\end{lem}

\begin{lem}$($\cite[Proposition 2.9]{Mat}$)$. \label{quot}
Let $\cT$ be a triangulated category and $\cU$ be its thick subcategory.
Denote by $Q: \cT \to \cT/\cK$ the canonical functor.
Then the map $Q^*: \Th(\cT/\cK) \to \Th(\cT)$ restricts to an immersion
$$
Q^*: \tSpec(\cT/\cK) \hookrightarrow \tSpec(\cT).
$$ 
of topological spaces whose image is $\{\cP \in \tSpec(\cT) \mid \cK \subseteq \cP\}$.
\end{lem}

The following theorem is one of the main results in \cite{Mat}, which is some sort of a generalization of Theorem \ref{balm}.

\begin{thm}\label{mat}
Let $X$ be a noetherian scheme.
\begin{enumerate}[\rm(1)]
\item
Let $\cP$ be a thick ideal of $\dpf(X)$.
Then $\cP$ is a prime thick subcategory if and only if $\cP$ is a prime thick ideal if and only if $\cP = \cS_X(x)$ for some $x \in X$.
\item
We have an immersion 
$$
\cS_X: X \hookrightarrow \tSpec(\dpf(X)), \quad x \mapsto \cS_X(x).
$$	
of topological spaces whose image is $\ttSpec(\dpf(X))$.
\end{enumerate}
\end{thm}

\begin{rem}
We will see in Corollary \ref{imm} that the above immersion $\cS_X: X \hookrightarrow \tSpec(\dpf(X))$ follows from more general result Theorem \ref{main2}.	
\end{rem}

From the definition, it is quite difficult to determine the topological space $\tSpec(\dpf(X))$ for a given noetherian scheme $X$.
For special cases, triangular spectra are determined as follows.

\begin{prop}\label{top}
\begin{enumerate}[\rm(1)]
\item $($\cite[Corollary 2.17(1)]{Mat}$)$.
If $X$ is a quasi-affine noetherian scheme, then there is a homeomorphism
$$
\tSpec(\dpf(X)) \cong X.
$$

\item $($\cite[Example 4.10]{Mat}$)$.
Let $\PP^1$ be the projective line over a field.
Then there is a homeomorphism
$$
\tSpec(\dpf(\PP^1)) \cong \PP^1 \sqcup \ZZ,
$$
where $\ZZ$ is considered as the discrete topological space.

\item $($\cite[Theorem 4.11]{HO}$)$.
Let $E$ be an elliptic curve over a field.
Then there is a homeomorphism
$$
\tSpec(\dpf(E)) \cong E \sqcup \left(\bigsqcup_{(r,d) \in I} E_{r,d}\right).
$$ 
Here, $I:=\{(r,d) \in \ZZ^{2} \mid r >0, \gcd(r,d) = 1\}$ and $E_{r,d}$ is a copy of $E$ for each $(r,d) \in I$.
\end{enumerate}
\end{prop}

\section{Reconstruction of schemes}

In this section, we discuss the reconstruction of a noetherian scheme $X$ from the triangulated category $\dpf(X)$. 
Reconstruction of underlying topological spaces has been discussed in \cite{HO, Mat} in terms of the triangular spectrum of $\dpf(X)$.
To reconstruct the structure sheaf, centers of triangulated categories play a crucial role.

We recall that a tensor triangulated category $(\cT, \otimes, \one)$ is {\it rigid} if it is {\it closed}, i.e., the exact functor $M \otimes -: \cT \to \cT$ has a right adjoint $[M,-]: \cT \to \cT$ for each $M \in \cT$ and such that every object $M$ is {\it strongly dualizable}, i.e., the canonical map
$$
[M, \one] \otimes N \to [M, N]
$$
is an isomorphism for each $N \in \cT$; see \cite{HPS} for details.
If $\cT$ is rigid, then so is $\cT(U)$ for any open subset $U \subseteq \ttSpec(\cT)$ by \cite[Proposition 2.15]{Bal07}. 

We say that the tensor triangulated category $(\cT, \otimes, \one)$ is {\it monogenic} if $\cT$ is classically generated by $\one$, i.e., $\cT= \langle \one \rangle$.
We say that $\cT$ is {\it locally monogenic} if, for any prime thick ideal $\cP \in \ttSpec(\cT)$, there is a quasi-compact open neighborhood $U \subseteq \ttSpec(\cT)$ of $\cP$ such that $\cT(U)$ is monogenic.

We begin with stating the following general result for specific tensor triangulated categories whose proof is taken from \cite[Lemma 4.10]{Rou}.

\begin{prop}\label{prop}
Let $(\cT, \otimes, \one)$ be an idempotent complete, rigid, locally monogenic tensor triangulated category. 
Then the evaluation at $\one$ induces an isomorphism
$$
\rZ(\cT)_{\lred} \cong \End_{\cT}(\one)_\red.
$$	
\end{prop}

\begin{proof}
As $\ttSpec(\cT)$ is a spectral space, there are quasi-compact open covering $U_1, U_2, \ldots, U_n$ of $\ttSpec(\cT)$ such that $\cT(U_i) = \langle \one_{U_i} \rangle$ for $i=1,2,\ldots, n$. 
	
Let $\alpha: \rZ(\cT) \to \End_{\cT}(\one), \eta \mapsto \eta_{\one}$ be the evaluation at $\one$.	
Since $\alpha$ has a right inverse $\End_{\cT}(\one) \to \rZ(\cT), \phi \mapsto \phi \otimes(-)$, the homomorphism $\alpha$ is surjective.
Therefore, it suffices to show that the induced homomorphism $\overline{\alpha}: \rZ(\cT)_{\lred} \to \End_{\cT}(\one)_\red$ is injective. 
To this end, let us prove $\eta \in \rZ(\cT)_\lnil$ for each $\eta \in \rZ(\cT)$ with $\alpha(\eta) = 0$, i.e., $\eta_\one = 0$.
We proceed by induction on $n$.

First assume $n=1$, i.e., $\cT = \langle \one \rangle$. 
We prove that $\eta_M$ is nilpotent for any $M \in \cT$.
We consider the subcategory $\cX \subseteq \cT$ consisting of objects $M \in \cT$ with $\eta_M$ nilpotent.
Then $\cX$ contains $\one$ as $\eta_\one = 0$.
In addition, $\cX$ is a thick subcategory.
Indeed, it is clear that $\cX$ is closed under direct summands and shifts. 
Take an exact triangle $L \xrightarrow{f} M \xrightarrow{g} N \to L[1]$ in $\cT$ with $L, N \in \cX$.
Then there is an integer $l \ge 1$ such that $\eta_L^{l} = 0$ and $\eta_N^{l}=0$.
From the naturality of $\eta^l$, the equality $g\eta_M^l = \eta_N^l g = 0$ holds.
Thus, $\eta_M^l$ factors as $\eta_M^l: M \xrightarrow{a} L \xrightarrow{f} M$.
Again using the naturality of $\eta^l$, we get $\eta_M^{2l} = \eta_M^{l} \eta_M^{l} = \eta_M^l fa = f \eta_L^l a = 0$.
As a result, $M \in \cX$ follows.
Therefore, $\cX$ is thick and so $\cX = \langle \one \rangle = \cT$.
This means that $\eta$ is locally nilpotent.

Next, assume that $n >1$ and set $V= U_2 \cup U_3 \cup \cdots U_n$.
Assume further that the evaluations at the unit objects induce isomorphisms
$$
\ol{\alpha}: \rZ(\cT(U_1))_\lred \xrightarrow{\cong} \End_{\cT(U_1)}(\one_{U_1})_\red,\quad \ol{\alpha}: \rZ(\cT(V))_\lred \xrightarrow{\cong}  \End_{\cT(V)}(\one_{V})_\red.
$$
Here, we note that there is a homeomorphism $f: \ttSpec(\cT(V)) \xrightarrow{\cong} V$ such that $\cT(V)(f^{-1}(U_i)) \cong \cT(U_i) = \langle \one_{U_i}\rangle$ for $i=2,3,\ldots, n$; see \cite[Proposition 1.11]{BF} and \cite[Constructions 24 and 29]{BalICM}. 
Hence we can apply the induction hypothesis to $\cT(V)$.

One can easily verify that the canonical functors $Q: \cT \to \cT/\cT^{U_1}$ and $\iota: \cT/\cT^{U_1} \to \cT(U_1)$ induce a commutative diagram
$$
\xymatrix{
\rZ(\cT)_\lred \ar[r]^-{\ol{Q_*}} \ar[d]_{\overline{\alpha}} & \rZ(\cT/\cT^{U_1})_\lred \ar[d]_{\overline{\alpha}} & \rZ(\cT(U_1))_\lred \ar[d]_{\overline{\alpha}}^{\cong} \ar[l]_-{\ol{\iota^*}}^\cong\\ 
\End_{\cT}(\one)_\red \ar[r]^-{\ol{Q}} & \End_{\cT/\cT^{U_1}}(Q(\one))_\red \ar[r]^-{\ol{\iota}}_\cong & \End_{\cT(U_1)}(\one_{U_1})_\red,
}
$$
where the evaluations at the unit objects induce the vertical arrows.
From this diagram, $\eta_\one = 0$ yields $Q_*(\eta) \in \rZ(\cT/\cT^{U_1})_\lnil$. 
For an object $M\in \cT$, there is an integer $l \ge 1$ such that $Q(\eta^l_M) = Q_*(\eta)^l_{Q(M)} =0$.
Accordingly, $\eta^l_M$ factors some $N \in \cT^{U_1}$.
Then $\eta_M^{(d+1)l}$ factors $\eta_N^{dl}$ for any integer $d \ge 1$.

For the canonical functors $Q': \cT \to \cT/\cT^V$ and $\iota': \cT/\cT^V \to \cT(V)$, the same argument as above shows that $Q'_*(\eta) \in \rZ(\cT/\cT^V)_\lnil$.
Here we use the isomorphism $\ol{\alpha}: \rZ(\cT(V))_\lred \xrightarrow{\cong} \End_{\cT(V)}(\one_{V})_\red$ which is our induction hypothesis.
In particular, $Q'(\eta _N^{dl}) = Q'_*(\eta)^{dl}_{Q'(N)}= 0$ for some $d\ge 1$.
It follows from \cite[Theorem 7.1]{BF} that the functor $Q'$ restricts to a fully faithful functor 
$$
Q': \cT^{U_1} \to (\cT/\cT^V)^{U_1}
$$
and hence $Q'(\eta _N^{dl}) = 0$ implies $\eta_N^{dl} = 0$.
Then $\eta^{(d+1)l}_M = 0$ holds as $\eta^{(d+1)l}_M$ factors $\eta_N^{dl}$.
As a result, we conclude that $\eta \in \rZ(\cT)_\lnil$.
\end{proof}

Let $X$ be a noetherian scheme and $U \subseteq X$ be an open subset with $Z:= X \setminus U$.
Then \cite[Theorem 2.13]{Bal02} shows that the restriction functor $(-)\lvert_U: \dpf(X) \to \dpf(U)$ induces a triangle equivalence
$$
\left(\dpf(X)/\mathrm{D^{pf}_{\mathit{Z}}}(X) \right)^\natural \cong \dpf(U)
$$
of tensor triangulated categories, where $\mathrm{D^{pf}_{\mathit{Z}}}(X) = \{F \in \dpf(X) \mid F\lvert_U \cong 0\} = \bigcap_{x \in U} \cS_X(x)$.
Therefore, the left-hand side is $\dpf(X)(\cS_X(U))$.
Applying Proposition \ref{prop} to $\cT= \dpf(X)$, the following result is recovered.

\begin{cor}$($\cite[Proposition 8.1]{Bal02} and \cite[Lemma 4.10]{Rou}$)$. \label{cntsect}
For a noetherian scheme $X$, there is an isomorphism
$$
\rZ(\dpf(X))_{\lred} \cong \Gamma(X, \cO_X)_\red.
$$	
\end{cor}

\begin{proof}
$\dpf(X)$ is idempotent complete because it is realized as the full subcategory of compact objects of $\dub_{qc}(\Mod \cO_X)$.
Also, it is rigid by \cite[Proposition 4.1]{Bal07}.	
For any affine scheme $U$, it holds that $\dpf(U) = \langle \cO_U \rangle$.   
It follows from Theorem \ref{balm} and \cite[Theorem 2.13]{Bal02} that the tensor triangulated category $\dpf(X)$ is locally monogenic. Thus, we get isomorphisms
$$
\rZ(\dpf(X))_{\lred} \cong \End_{\dpf(X)}(\cO_X)_\red \cong\Gamma(X, \cO_X)_\red
$$	
of commutative rings from Proposition \ref{prop}.
\end{proof}

Now we state and prove the first main result in this paper about the reconstruction of a noetherian scheme from the perfect derived category.

\begin{thm}\label{main}
Let $X_1$ and $X_2$ be noetherian schemes and $\Phi: \dpf(X_1) \xrightarrow{\cong} \dpf(X_2)$ be a triangle equivalence.
Assume that there is a line bundle $L_i$ on $X_i$ for $i=1,2$ satisfying the following conditions:
\begin{enumerate}[\rm(i)]
\item
$L_i$ is ample or anti-ample for $i=1,2$.
\item
There is an isomorphism
$$
\Phi(F \ltensor_{\cO_{X_1}} L_1) \cong \Phi(F) \ltensor_{\cO_{X_2}} L_2
$$	
for any $F \in \dpf(X_1)$.
\end{enumerate}
Then $(X_1)_{\red}$ and $(X_2)_{\red}$ are isomorphic as schemes.
\end{thm}

\begin{proof}
Let us first remark that as $L_i$ or $L_i^{\otimes -1}$ is ample, $\{L_i^{\otimes m} \mid m \in \ZZ\}$ generates $\dpf(X_i)$ by \cite[Lemma 2.2]{Nee1} and \cite[Example 1.10]{Nee2}.

For $i=1,2$, denote by $\tSpec^{L_i}(\dpf(X_i))$ the subset of $\tSpec(\dpf(X_i))$ consisting of prime thick subcategories $\cP$ with $\cP \ltensor_{\cO_{X_i}} L_i := \{F \ltensor_{\cO_{X_i}} L_i \mid F \in \cP\} = \cP$. 
Then the homeomorphism $\Phi^*: \tSpec(\dpf(X_2)) \xrightarrow{\cong} \tSpec(\dpf(X_1))$ restricts to the homeomorphism $\Phi^*: \tSpec^{L_2}(\dpf(X_2)) \xrightarrow{\cong} \tSpec^{L_1}(\dpf(X_1))$ by (ii).
On the other hand, by the first remark, $\{L_i^{\otimes n} \mid n \in \ZZ\}$ generates $\dpf(X_i)$ and hence Theorem \ref{mat}(1) implies $\tSpec^{L_i}(\dpf(X_1)) = \ttSpec(\dpf(X_i))$.
From Theorem \ref{balm}, we obtain a homeomorphism
$
\cS_{X_i}: X_i \xrightarrow{\cong} \ttSpec(\dpf(X_1)) = \tSpec^{L_i}(\dpf(X_i)).
$
Thus, we get a homeomorphism
$$
f: X_2 \xrightarrow{\cS_{X_2}} \tSpec^{L_2}(\dpf(X_2)) \xrightarrow{\Phi^*} \tSpec^{L_1}(\dpf(X_1)) \xrightarrow{\cS_{X_1}^{-1}} X_1.
$$
The construction of this map yields $\Phi(\cS_{X_1}(f(x_2))) = \cS_{X_2}(x_2)$ for any $x_2 \in X_2$.

For an open subset $U \subseteq X_2$ with $Z = X_2 \setminus U$, one has
$$
\Phi(\mathrm{D^{pf}_{\mathit{f(Z)}}}(X_1)) = \Phi\left(\bigcap_{x_1 \in f(Z)}\cS_{X_1}(x_1)\right) = \bigcap_{x_1 \in f(Z)} \Phi\left(\cS_{X_1}(x_1)\right) = \bigcap_{x_2 \in Z} \cS_{X_2}(x_2) = \mathrm{D^{pf}_{\mathit{Z}}}(X_2).
$$
Therefore, the triangle equivalence $\Phi: \dpf(X_1) \xrightarrow{\cong} \dpf(X_2)$ induces a triangle equivalence
$$
\ol{\Phi}: \dpf(X_1)/\mathrm{D^{pf}_{\mathit{f(Z)}}}(X_1) \xrightarrow{\cong} \dpf(X_2)/\mathrm{D^{pf}_{\mathit{Z}}}(X_2).
$$
Taking idempotent completion and $\rZ(-)_{\lred}$, we obtain an isomorphism 
$$
\Gamma(f(U), \cO_{X_1})_\red \cong \Gamma(U, \cO_{X_2})_\red
$$
by \cite[Theorem 2.13]{Bal02} and Corollary \ref{cntsect}. 
As we can easily see that this isomorphism is compatible with the restriction of open subsets, we get an isomorphism $(X_2)_\red \cong (X_1)_\red$ of schemes.
\end{proof}

This result has several applications.
The first one is the following reconstruction theorem which is well-known at least for the affine case.

\begin{cor}\label{qaff}
Let $X_1$ and $X_2$ be noetherian quasi-affine schemes.
If there is a triangle equivalence $\Phi: \dpf(X_1) \xrightarrow{\cong} \dpf(X_2)$, then $(X_1)_{\red}$ and $ (X_2)_{\red}$ are isomorphic as schemes.
\end{cor}

\begin{proof}
Take $L_i$ to be $\cO_{X_i}$.
Then $\cO_{X_i}$ is an ample line bundle as $X_i$ is quasi-affine.
Moreover, the isomorphism 
$$
\Phi(F \ltensor_{\cO_{X_1}} \cO_{X_1}) \cong \Phi(F) \ltensor_{\cO_{X_2}} \cO_{X_2}
$$
obviously holds for each $F \in \dpf(X_1)$; hence, the result follows by Theorem \ref{main}.
\end{proof}
 
Another application of Theorem \ref{main} is to recover the famous result by Bondal-Orlov \cite{BO} and Ballard \cite{Ball}. 

\begin{cor}$($\cite[Theorem 2.5]{BO}, \cite[Theorem 1]{Ball}$)$. \label{BOB}
Let $X_i$ be a Gorenstein projective scheme over a field $k$ with ample or anti-ample canonical bundle $\omega_{X_i}$ for $i=1,2$.
If there is a $k$-linear triangle equivalence $\Phi: \dpf(X_1) \xrightarrow{\cong} \dpf(X_2)$, then $(X_1)_{\red}$ and $(X_2)_{\red}$ are isomorphic as schemes.
\end{cor}

\begin{proof}
Take $L_i$ to be $\omega_{X_i}$.
Then the assumption (i) is satisfied.
Recall that a Serre functor $S_{X_i}: \dpf(X_i) \to \dpf(X_i)$ on $\dpf(X_i)$ is a triangle equivalence such that there is a natural isomorphism
$$
\Hom_{\dpf(X_i)}(F, G) \cong \Hom_{\dpf(X_i)}(G, S_{X_i}(F))^*
$$
for any $F, G \in \dpf(X_i)$.
Here $(-)^*$ stands for the $k$-dual functor.
It follows from \cite[Lemma 6.6]{Ball} that $S_{X_i} \cong (-) \ltensor_{\cO_{X_i}} \omega_{X_i}[-\dim X_i]$ and from \cite[Lemma 1.30]{Huy} that $S_{X_2} \circ \Phi \cong \Phi \circ S_{X_1}$.
Therefore, the assumption (ii) follows.
Applying Theorem \ref{main}, we get an isomorphism between $(X_1)_{\red}$ and $(X_2)_{\red}$.
\end{proof}



\section{Triangular spectra as ringed spaces}

One of the key ingredients of proof of Theorem \ref{main} is the isomorphisms
$$
\rZ\left(\left(\dpf(X)/ \bigcap_{x \in U} \cS_X(x)\right)^\natural\right)_\lred \cong \rZ\left(\dpf(U)\right)_\lred \cong \Gamma(U, \cO_X)_\red \quad 
$$
for an open subset $U \subseteq X$, which is a consequence of \cite[Theorem 2.13]{Bal02} and Corollary \ref{cntsect}.
Motivated by these isomorphisms, we define the structure sheaf on the triangular spectrum for a triangulated category.

Let $\cT$ be a triangulated category.
For an open subset $U \subseteq \tSpec(\cT)$, we set
$$
\cT^U := \bigcap_{\cP \in U} \cP, \quad \cT(U) := \left(\cT/\cT^U \right)^\natural.
$$
By composing the quotient functor $Q_{U}: \cT \to \cT/\cT^U$ and the idempotent completion functor $\iota_U: \cT/\cT^U \to \cT(U) $, we have a natural functor $\res_U: \cT \to \cT(U)$. 
Thanks to Lemmas \ref{cntidm} and \ref{cntquot}, the functor $\res_U$ induces a homomorphism
$$
(\res_U)_*: \rZ(\cT)_\lred \xrightarrow{\ol{(Q_U)_*}} \rZ(\cT/\cT^U)_\lred \underset{\cong}{\xrightarrow{((\ol{\iota_U)^*})^{-1}}} \rZ(\cT(U))_\lred.
$$
Let $U \subseteq V \subseteq \tSpec(\cT)$ be open subsets.
Then the universal properties of the Verdier quotient and the idempotent completion shows that the inclusion $\cT^V \subseteq \cT^U$ induces a unique exact functor 
$$
\res_{V,U}: \cT(V) \to \cT(U)
$$
such that $\res_{V,U} \circ \res_V = \res_U$.
Again using Lemmas \ref{cntidm} and \ref{cntquot}, this functor induces a homomorphism 
$$
(\res_{V,U})_* : \rZ(\cT(V))_\lred \to \rZ(\cT(U))_\lred.
$$
Furthermore, we have the equality $(\res_{W,V})_* \circ (\res_{V,U})_* = (\res_{W,U})_*$ for three open subsets $U \subseteq V \subseteq W \subseteq \tSpec(\cT)$. 

\begin{dfn}
Let $\cT$ be a triangulated category.	
We define the sheaf $\cO_{\triangle, \cT}$ of commutative rings on $\tSpec(\cT)$ by the sheafification of the presheaf $\cO^p_{\triangle, \cT}$:
$$
\cO^p_{\triangle, \cT}(U) := \rZ(\cT(U))_{\lred}, \quad
\rho_{V,U}:= (\res_{V,U})_{*}:  \cO^p_{\triangle, \cT}(V) \to \cO^p_{\triangle, \cT}(U).
$$
Later, we consider the spectrum $\tSpec(\cT)$ as the ringed space $(\tSpec(\cT), \cO_{\triangle, \cT})$.	
\end{dfn}

Recall that a morphism $(f, f^\flat): (X, \cO_X) \to (Y, \cO_Y)$ of ringed spaces consists of a continuous map $f: X \to Y$ and a morphism $f^\flat: f^{-1}\cO_Y \to \cO_X$ of sheaves of commutative rings on $X$.
We say that the morphism $(f, f^\flat)$ is an {\it open immersion} of ringed spaces if $f:X \to Y$ is an open immersion of topological spaces and $f^\flat: f^{-1}\cO_Y \to \cO_X$ is an isomorphism.
We remark that for a continuous map $f: X \to \tSpec(\cT)$, the pullback sheaf $f^{-1}\cO_{\triangle, \cT}$ is isomorphic to the sheafification of the presheaf
$$
f^p\cO^p_{\triangle, \cT}: U \mapsto \ilim[f(U) \subseteq V]{\cO^p_{\triangle, \cT}(V)}
$$
of $X$.

Lemmas \ref{cntidm} and \ref{cntquot} show that the continuous maps in Lemmas \ref{idm} and \ref{quot} can be extended to morphisms of ringed spaces.

\begin{prop}\label{rsidm}
Let $\Phi: \cT \to \cT'$ be a fully faithful dense exact functor.
Then we have an isomorphism 
$$
\Phi^*: \tSpec(\cT') \xrightarrow{\cong} \tSpec(\cT)
$$	
of ringed spaces.
\end{prop}

\begin{proof}
This follows from \ref{cntidm} and Lemmas \ref{idm}.
\end{proof}

\begin{prop}\label{rsquot}
Let $\cU$ be a thick subcategory of $\cT$.
Then the quotient functor $\Phi: \cT \to \cT/\cU$ induces a morphism
$$
\Phi^*: \tSpec(\cT/\cK) \to \tSpec(\cT)
$$	
of ringed spaces.
Moreover, if the image of $\Phi^*$ is an open subset of $\tSpec(\cT)$, then the above morphism is an open immersion of ringed spaces.
\end{prop}

\begin{proof}
We construct a morphism $(\Phi^*)^\flat: (\Phi^*)^{-1}\cO_{\triangle, \cT} \to \cO_{\triangle, \cT/\cK}$ of sheaves of commutative rings.
Take open subsets $U \subseteq \tSpec(\cT/\cK)$ and $V \subseteq \tSpec(\cT)$ with $\Phi^*(U) \subseteq V$. 
Setting $\widetilde{U} := \Phi^*(U) = \{\cP \in \tSpec(\cT) \mid \cK \subseteq \cP,  \cP/\cK \in U\}$, one sees $U= \{\cP/\cK \mid \cP \in \widetilde{U}\}$ and hence $(\cT/\cK)^U = \bigcap_{\cP \in \widetilde{U}}\cP/\cK = \cT^{\widetilde{U}}/\cK$.
Then the inclusion $\cT^V \subseteq \cT^{\widetilde{U}}$ induces a triangle functor
$$
\cT/\cT^V \to \cT/\cT^{\widetilde{U}} \cong (\cT/\cK)/(\cT^{\widetilde{U}}/\cK) = (\cT/\cK)/(\cT/\cK)^U,
$$
where the triangle equivalence is by \cite[Proposition 2.3.1(c)]{Ver}.
Applying $\rZ((-)^\natural)$, we get a homomorphism $\cO_{\triangle, \cT}^p(V) \to \cO_{\triangle, \cT/\cK}^p(U)$.
Moreover, let $U \subseteq U' \subseteq \tSpec(\cT/\cK)$ and $V \subseteq V' \subseteq \tSpec(\cT)$ be other open subsets with $\widetilde{U'} :=\Phi^*(U') \subseteq  V'$.
Then the inclusions
$$
\xymatrix{
\cT^V  \ar@{^{(}-_>}[r] & \cT^{\widetilde{U}} \\
\cT^{V'}  \ar@{^{(}-_>}[r]  \ar@{^{(}-_>}[u] & \cT^{\widetilde{U'}}  \ar@{^{(}-_>}[u]
}
$$
induce a commutative diagram
$$
\xymatrix{
\cT/\cT^V  \ar[r] & \cT/\cT^{\widetilde{U}} \ar@{}[r]|-*[@]{\cong} & (\cT/\cK)/(\cT/\cT^{\widetilde{U}})\\
\cT/\cT^{V'}  \ar[r]  \ar[u] & \cT/\cT^{\widetilde{U'}}  \ar[u] \ar@{}[r]|-*[@]{\cong} & (\cT/\cK)/(\cT/\cT^{\widetilde{U'}}). \ar[u]
}
$$
Applying $\rZ((-)^\natural)_\lred$, we obtain a commutative diagram
$$
\xymatrix{
\cO_{\triangle, \cT}^p(V)  \ar[r] & \cO_{\triangle, \cT/\cK}^p(U) \\
\cO_{\triangle, \cT}^p(V')  \ar[r]  \ar[u] & \cO_{\triangle, \cT/\cK}^p(U'). \ar[u]
}
$$
Therefore, the homomorphism $\cO_{\triangle, \cT}^p(V) \to \cO_{\triangle, \cT/\cK}^p(U)$ defines a morphism 
$
(\Phi^*)^p\cO_{\triangle, \cT}^p \to \cO_{\triangle, \cT/\cK}^p
$
of presheaves of commutative rings.
Taking sheafifications, we obtain a morphism $(\Phi^*)^\flat: (\Phi^*)^{-1}\cO_{\triangle, \cT} \to \cO_{\triangle, \cT/\cK}$ of sheaves of commutative rings.

Next, assume $\Phi^*(\tSpec(\cT/\cK)) \subseteq \tSpec(\cT)$ is an open subset.
Then $\Phi^*: \tSpec(\cT/\cK) \to \tSpec(\cT)$ is an open immersion of topological spaces.
For any open subset $U \subseteq \tSpec(\cT/\cK)$, the map
$
(\Phi^*)^p \cO^p_{\triangle, \cT/\cK}(U) \to  \cO^p_{\triangle, \cT}(U)  
$
is induced from the triangle equivalence
$$
\cT/\cT^{\widetilde{U}} \cong (\cT/\cK)/(\cT^{\widetilde{U}}/\cK)  = (\cT/\cK)/(\cT/\cK)^U.
$$
As a result, the map 
$
(\Phi^*)^p\cO^p_{\triangle, \cT/\cK} \to \cO^p_{\triangle, \cT}
$
is an isomorphism and hence so is $
(\Phi^*)^\flat: (\Phi^*)^{-1}\cO_{\triangle, \cT/\cK} \to \cO_{\triangle, \cT}
$.
\end{proof}

\begin{cor}\label{princ}
For an object $M \in \cT$, the quotient functor $\Phi: \cT \to \cT/\langle M \rangle$ induces an open immersion 
$$
\Phi^*: \tSpec(\cT/\langle M \rangle) \hookrightarrow \tSpec(\cT)
$$
of ringed spaces.
\end{cor}

\begin{proof}
The image of $\Phi^*$ is
$$
\{\cP \in \tSpec(\cT) \mid M \in \cT\} = \tSpec(\cT) \setminus \sZ(\{M\}),
$$	
which is an open subset of $\tSpec(\cT)$.
The result follows from Proposition \ref{rsquot}.
\end{proof}

The main theorem in this section compares the Balmer spectrum and the triangular spectrum for an idempotent complete, rigid, and locally monogenic tensor triangulated category.
We treat the monogenic case first.

\begin{lem}\label{aff}
Let $\cT$ be an idempotent complete, rigid, monogenic tensor triangulated category.
Assume further that $\ttSpec(\cT)$ is noetherian.
Then the equality $\tSpec(\cT) = \ttSpec(\cT)$ holds.
\end{lem}

\begin{proof}
Since $\cT = \langle \one \rangle$, every thick subcategory is a thick ideal.
Therefore, by \cite[Remark 4.3]{Bal05}, the thick subcategories and the radical thick ideals coincide.

Let $\cP \in \tSpec(\cT)$ and take the smallest element $\cX$ in $\{\cX \in \Th(\cT) \mid \cP\subsetneq \cX\}$.
Since $\cP$ is a radical thick ideal by the above argument, it is the intersection of all prime thick ideals $\cQ$ containing $\cP$ by \cite[Lemma 4.2]{Bal05}.
If $\cP$ is not a prime thick ideal, then such $\cQ$ satisfies $\cP \subsetneq \cQ$ and hence $\cX \subseteq \cQ$ by the assumption on $\cX$.
Therefore, $\cP \subsetneq \cX \subseteq \bigcap_{\cQ \in \ttSpec(\cT), \cP \subseteq \cQ} \cQ = \cP$ leads a contradiction.
Hence we conclude that the inclusion $\tSpec(\cT) \subseteq \ttSpec(\cT)$ holds.

Conversely, take $\cP \in \ttSpec(\cT)$ and prove $\cP \in \tSpec(\cT)$. 
It follows from \cite[Theorem 4.10]{Bal05} that there is an order-preserving bijection between the thick subcategories of $\cT$ and the specialization-closed subsets of $\ttSpec(\cT)$. 
The noetherian assumption is used here.
Under this bijection, $\cP$ corresponds to the specialization-closed subset
$$
W := \{\cQ \in \ttSpec(\cT) \mid \cP \not\subseteq \cQ \} = \{\cQ \in \ttSpec(\cT) \mid \cP \not\in \ol{\{\cQ\}} \}.
$$ 
Here we use \cite[Proposition 2.9]{Bal05}.
Therefore, it suffices to show that there is the smallest specialization-closed subset $T$ with $W \subsetneq T$.
Set $T: = W \cup \{\cP\}$ and prove that this is the specialization-closed subset that we need.
For an element $\cQ \in \ol{\{\cP\}}\setminus \{\cP\}$, $\cQ$ belongs to $W$ as $\cP \not\subseteq \cQ$.
Thus $T = W \cup \ol{\{\cP\}}$ holds and this is specialization-closed.
For a specialization-closed subset $W' \subseteq \ttSpec(\cT)$ with $W \subsetneq W'$, we will check $T \subseteq W'$.
Since $W \subsetneq W'$, there is an element $\cQ \in W'$ such that $\cQ \not\in W$.
Then $\cQ \not\in W$ implies $\cP \in \ol{\{\cQ\}}$.
As $W'$ is specialization-closed, one has $\cP \in \ol{\{\cQ\}} \subseteq W'$.
This shows that the inclusion $T \subseteq W'$ and hence $T$ is the smallest specialization-closed subset with $W \subsetneq T$.
\end{proof}

\begin{thm}\label{main2}
Let $(\cT, \otimes, \one)$ be an idempotent complete, rigid, locally monogenic tensor triangulated category. 
Assume further that $\ttSpec(\cT)$ is noetherian.
\begin{enumerate}[\rm(1)]
\item
Let $\cP$ be a thick ideal of $\cT$. 
Then $\cP$ is a prime thick ideal if and only if it is a prime thick subcategory.
In particular, there is an inclusion $\ttSpec(\cT) \subseteq \tSpec(\cT)$.
\item
There is a morphism 
$$
i: \ttSpec(\cT)_\red \to \tSpec(\cT)
$$
of ringed spaces satisfying the following conditions:
\begin{enumerate}[\rm(a)]
\item
The map of underlying topological spaces is the inclusion $\ttSpec(\cT) \subseteq \tSpec(\cT)$.
\item
$i$ is an open immersion of ringed spaces whenever $\ttSpec(\cT)$ is an open subset of $\tSpec(\cT)$.
\item
$i$ is an isomorphism of ringed spaces if $\cT$ is monogenic.
\end{enumerate}
	
\end{enumerate}
\end{thm}

\begin{proof}
(1) The ``only if" part follows from \cite[Proposition 4.8]{Mat}. 
Here we note that under the assumption, every thick ideal is radical by \cite[Proposition 2.4]{Bal07}.

Let us show the ``if" part.
Take $\cP \in \ttSpec(\cT)$ and a quasi-compact open neighborhood $\cP \in U \subseteq \ttSpec(\cT)$ with $\cT(U) = \langle \one_U \rangle$. 
Then $\cT(U)$ is idempotent complete, rigid, and monogenic.
On the other hand, it follows from \cite[Proposition 1.11]{BF} that the canonical functor $\res_U: \cT \to \cT(U)$ induces a homeomorphism $(\res_U)^*: \ttSpec(\cT(U)) \xrightarrow{\cong} U$.
In particular, $\ttSpec(\cT(U))$ is noetherian.
Applying Lemma \ref{aff} yields the equality $\ttSpec(\cT(U)) = \tSpec(\cT(U))$.
Then $U$ is a subset of $\tSpec(\cT)$ via the composition
$$
U \xrightarrow{((\res_U)^*)^{-1}} \ttSpec(\cT(U)) = \tSpec(\cT(U)) \xrightarrow{(\res_U)^*} \tSpec(\cT).
$$
Therefore, we get $\cP \in U \subseteq \tSpec(\cT)$.

(2) 
Let $i: \ttSpec(\cT) \hookrightarrow \tSpec(\cT)$ denote the inclusion and construct a morphism $i^{-1}\cO_{\triangle, \cT} \to (\cO_{\otimes, \cT})_\red$ of sheaves of rings.
For open subsets $U \subseteq \ttSpec(\cT)$ and $V \subseteq \tSpec(\cT)$ with $U \subseteq V$, the inclusion $\cT^V \subseteq \cT^U$ induces a  homomorphism
$$
\cO^p_{\triangle, \cT}(V) = \rZ\left(\cT(V) \right)_\lred \to \rZ\left(\cT(U) \right)_\lred \cong \End_{\cT(U)}(\one_U)_\red = \cO^p_{\otimes, \cT}(U)_\red,
$$
where the isomorphism is proved in Proposition \ref{prop}.
Moreover, for other open subsets $U \subseteq U' \subseteq \ttSpec(\cT)$ and $V \subseteq V' \subseteq \tSpec(\cT)$ with $U' \subseteq V'$, we get the commutative diagram
$$
\xymatrix{
\cT/\cT^V  \ar[r] & \cT/\cT^{U} \\
\cT/\cT^{V'}  \ar[r]  \ar[u] & \cT/\cT^{U'}. \ar[u]
}
$$
Applying $\rZ((-)^\natural)_\lred$ yields a commutative diagram
$$
\xymatrix{
\cO_{\triangle, \cT}^p(V) \ar@{=}[r] & \rZ(\cT(V))_\lred  \ar[r] & \rZ(\cT(U))_\lred \ar@{}[r]|-*[@]{\cong} & \cO^p_{\otimes, \cT}(U)_\red \\
\cO_{\triangle, \cT}^p(V') \ar@{=}[r] \ar[u] & \rZ(\cT(V'))_\lred  \ar[r]  \ar[u] & \rZ(\cT(U'))_\lred \ar[u] \ar@{}[r]|-*[@]{\cong} & \cO^p_{\otimes, \cT}(U')_\red. \ar[u]
}
$$
Therefore, $\cO^p_{\triangle, \cT}(V) \to  \cO^p_{\otimes, \cT}(U)_\red$ defines a morphism
$
i^p\cO^p_{\triangle, \cT} \to (\cO^p_{\otimes, \cT})_\red
$
of presheaves of commutative rings.
Taking sheafifications, we obtain a morphism 
$
i^\flat: i^{-1}\cO_{\triangle, \cT} \to (\cO_{\otimes, \cT})_\red 
$
of sheaves of commutative rings.

Assume $\ttSpec(\cT)$ is an open subset of $\tSpec(\cT)$.
Then for any open subset $U \subseteq \ttSpec(\cT)$,  the homomorphism $i^p\cO^p_{\triangle, \cT}(U) \to \cO^p_{\otimes, \cT}(U)$ is given by the isomorphism
$$
\rZ(\cT(U))_\lred \cong \End_{\cT(U)}(\one_U)_\red
$$
in Proposition \ref{prop}.
Thus, $i^p\cO^p_{\triangle, \cT}(U) \to \cO^p_{\otimes, \cT}(U)$ is an isomorphism and hence so is $i^{-1}\cO_{\triangle, \cT}(U) \to \cO_{\otimes, \cT}(U)$. 
This means that $i: \ttSpec(\cT) \to \tSpec(\cT)$ is an open immersion of ringed spaces.
In particular, if $\cT$ is monogenic, then $\ttSpec(\cT) = \tSpec(\cT)$ holds by \ref{aff} and hence $i: \ttSpec(\cT) \to \tSpec(\cT)$ is an isomorphism of ringed spaces.
\end{proof}

As $\dpf(X)$ is an idempotent complete, rigid, locally monogenic tensor triangulated category, we get the following immediate consequence of the combination of Theorems \ref{balm} and \ref{main2}.

\begin{cor}\label{imm}
For a noetherian scheme $X$, there is a morphism of ringed spaces
$$
\cS_X: X_{\mathrm{red}} \to \tSpec(\dpf(X)),\quad x \mapsto \cS_X(x),
$$
which is an open immersion of ringed spaces if $\cS_X(X)$ is open in $\tSpec(\dpf(X))$.
In particular, $\cS_X: X_{\mathrm{red}} \to \tSpec(\dpf(X))$ is an isomorphism of ringed spaces if $X$ is quasi-affine.
\end{cor}

\begin{rem}
Since the ringed space $\tSpec(\dpf(X))$ is entirely determined by the triangulated category structure of $\dpf(X)$, Corollary \ref{imm} directly implies Corollary \ref{qaff}.	
Moreover, in a recent work \cite{IM}, Ito and the author proved that $\cS_X(X) \subseteq \tSpec(\dpf(X))$ is an open subset if $X$ is a quasi-projective variety over an algebraically closed field. 
Consequently, $\cS_X: X \to \tSpec(\dpf(X))$ is an open immersion of ringed spaces under the assumption.
As an application, we relax the assumption (i) in Theorem \ref{main}, and it gives an alternative proof of Bondal-Orlov, Ballad reconstruction theorem \cite[Theorem 2]{Ball}. 
\end{rem}

Using Corollary \ref{imm}, we determine $\tSpec(\dpf(X))$ for $X$ which appeared in Proposition \ref{top}.

\begin{cor}\label{p1}
Let $\PP^1$ be the projective line over a field $k$.
Then there is an isomorphism 
$$
\tSpec(\dpf(\PP^1)) \cong \PP^1 \sqcup \left( \bigsqcup_{n \in \ZZ} \Spec (k)\right)
$$	
of ringed spaces.
\end{cor}

\begin{proof}
For any integer $n \in \ZZ$, it follows from \cite[Corollary 8.29]{Huy} that we get triangle equivalences $\dpf(\PP^1)/\langle \cO_{\PP^1}(n)\rangle \cong \langle \cO_{\PP^1}(n+1) \rangle \cong \dpf(\Spec(k))$.
It follows from Corollaries \ref{princ} and \ref{imm} that there is an open immersion
$$
f_n: \Spec(k) \underset{\cong}{\xrightarrow{\cS_{\Spec(k)}}}\tSpec(\dpf(\Spec(k))) \cong \tSpec(\dpf(\PP^1)/ \langle \cO_{\PP^1}(n) \rangle) \hookrightarrow \tSpec(\dpf(\PP^1)) 
$$	
of ringed spaces whose image is $\langle \cO_{\PP^1}(n) \rangle$.
Moreover, Corollary \ref{imm} and \cite[Corollary 4.7]{HO} show that there is an open immersion
$$
g:=\cS_{\PP^1}: \PP^1 \hookrightarrow \tSpec(\dpf(\PP^1)) 
$$
of ringed spaces.
As it is proved in \cite[Example 4.10]{Mat}, $\tSpec(\dpf(\PP^1))$ is the disjoint union of the images of $f_n\,\,(n \in \ZZ)$ and $g$.
Therefore, $f_n\,\,(n \in \ZZ)$ and $g$ induce isomorphism
$$
\PP^1 \sqcup \left( \bigsqcup_{n \in \ZZ} \Spec (k)\right) \cong \tSpec(\dpf(\PP^1)) 
$$
of ringed spaces.
\end{proof}

\begin{cor}\label{elp}
Let $E$ be an elliptic curve over an algebraically closed field. Then there is an isomorphism
$$
\tSpec(\dpf(E)) \cong E \sqcup \left(\bigsqcup_{(r,d) \in I} E_{r,d}\right)
$$
of ringed spaces.
Here, $I:=\{(r,d) \in \ZZ^{2} \mid r >0, \gcd(r,d) = 1\}$ and $E_{r,d}$ is a copy of $E$ for each $(r,d) \in I$.

In particular, for elliptic curves $E$ and $E'$, if there is a triangle equivalence $\dpf(E) \cong \dpf(E')$, then there is an isomorphism $E \cong E'$ of schemes.
\end{cor}

\begin{proof}
For an element $(r,d) \in I$, $M(r,d)$ denotes the moduli space of $\mu$-semistable sheaves with Chern character $(r,d)$. 
It follows from \cite[Theorem 1]{Tu} that $M(r,d) \cong E_{r,d}$, where $E_{r,d}$ is a copy of $E$. 
Moreover, \cite[Proposition 3]{HP} shows that there is a triangle equivalence 
$$
\Phi_{r,d}:  \dpf(E)\xrightarrow{\cong} \dpf(M(r,d)).
$$
Then we get open immersions 
\begin{align*}
f_{r,d}&: E_{r,d} \cong M(r,d) \overset{\cS_{M(r,d)}}{\hookrightarrow} \tSpec(\dpf(M(r,d))) \underset{\cong}{\xrightarrow{\Phi_{r,d}^*}} \tSpec(\dpf(E)), \\	
g &:= \cS_E: E \hookrightarrow \tSpec(\dpf(E))
\end{align*}
of ringed spaces by Corollary \ref{imm} and \cite[Corollary 4.7]{HO}.
It is shown in  \cite[Theorem 4.11]{HO} that $\tSpec(\dpf(E))$ is the disjoint union of open subsets $f_{r,d}(E_{r,d})\,\, ((r,d) \in I)$ and $g(E)$.
Hence $f_{r,d}\,\, ((r,d) \in I)$ and $g$ induce an isomorphism
$$
E \sqcup \left(\bigsqcup_{(r,d) \in I} E_{r,d}\right) \xrightarrow{\cong} \tSpec(\dpf(E))
$$
of ringed spaces. 

Next, assume there is a triangle equivalence $\dpf(E) \cong \dpf(E')$ for elliptic curves $E$ and $E'$.
Then the first statement shows that there is an isomorphism between disjoint unions of copies of $E$ and $E'$: 
$$
\bigsqcup E  \cong \tSpec(\dpf(E)) \cong \tSpec(\dpf(E')) \cong  \bigsqcup E'. 
$$
Comparing their connected components, we get an isomorphism $E \cong E'$.
\end{proof}

\begin{rem}
The second statement of Corollary \ref{elp} can be found in \cite[Pages 134 and 135]{Huy} for example.
However, our proof uses the completely different argument.
\end{rem}


\end{document}